\documentclass[12pt,reqno]{amsart}
\usepackage{amsmath, amssymb, latexsym, amscd, amsthm,amsfonts,amstext}
\usepackage[mathscr]{eucal}
\usepackage{xspace}
\usepackage{array, tabularx, longtable}
\usepackage{multicol,color}
\usepackage{indentfirst}
\usepackage{graphics}
\usepackage[colorlinks=true]{hyperref}
\theoremstyle{plain}
\newtheorem{theorem}{Theorem}

\newtheorem{lemma}[theorem]{Lemma}
\newtheorem{proposition}[theorem]{Proposition}
\theoremstyle{remark}

\newtheorem{claim}{Claim}

\theoremstyle{definition}

\begin{document}
\date{} 

\title [Gauss map of complete minimal surfaces on annular ends ]{Ramification of the Gauss map of complete minimal surfaces in $\mathbb R^m$ on annular ends}

\author[G. Dethloff, P. H. Ha and P. D. Thoan]{Gerd Dethloff, Pham Hoang Ha and Pham Duc Thoan}


\begin{abstract} 
In this article, we study the ramification of the Gauss map of complete minimal surfaces in $\mathbb R^m$ on annular ends. This work is a continuation of previous work of Dethloff-Ha (\cite{DH}). We thus give an improvement of the results on annular ends of complete minimal surfaces of Jin-Ru (\cite{JR}). 
\end{abstract}
\maketitle
\tableofcontents

\section{Introduction}
In 1988, H. Fujimoto (\cite{Fu}) proved Nirenberg's conjecture that if $M$ is a complete non-flat minimal surface in $\mathbb R^3,$ then its Gauss map can omit at most 4 points, and the bound is sharp. After that, he also extended that result for  minimal surfaces in $\mathbb R^m.$ He proved that the Gauss map of a non-flat complete minimal surface can  omit at most $m(m+1)/2$ hyperplanes in $\mathbb P^{m-1}(\mathbb C)$ located in general position (\cite{Fu2}). He also gave an example to show that the number $m(m+1)/2$ is the best possible when $m$ is odd
(\cite{Fu3}).\\

In 1993, M. Ru (\cite{Ru2}) refined these results  by studying the Gauss maps of minimal surfaces in $\mathbb R^m$ with ramification. Using the notations 
which will be introduced in $\S 3,$ the result of Ru can be stated as follows.\\

\noindent {\bf Theorem A.}\ {\it Let $M$ be a non-flat complete minimal surface in $\mathbb R^m$.
Assume that the (generalized) Gauss map $g$ of $M$ is $k-$non-degenerate (that is $g(M)$ is contained in a $k-$dimensional linear subspace in $ \mathbb P^{m-1}(\mathbb C)$, but none of lower dimension), $1\leq k \leq m-1$.
Let  $\{H_j\}_{j=1}^q$ be hyperplanes in general position in $ \mathbb P^{m-1}(\mathbb C)$ such that $g$ is ramified over
$H_j$ with multiplicity at least $m_j$  for each j, then 
$$
\sum_{j=1}^q(1 - \frac{k}{m_j})\leq (k+1)(m-\dfrac{k}{2}-1)+m\,. \   
$$
In particular  if there are q $(q > m(m+1)/2)$ hyperplanes  $\{H_j\}_{j=1}^q$ in general position in $ \mathbb P^{m-1}(\mathbb C)$ such that $g$ is ramified over
$H_j$ with multiplicity at least $m_j$  for each j, then 
$$\sum_{j=1}^q(1 - \frac{m-1}{m_j})\leq \dfrac{m(m+1)}{2}\,. \   $$}

On the other hand, in 1991, S. J. Kao (\cite{Kao}) used the ideas of Fujimoto~(\cite{Fu}) to show that the Gauss map of an end of a non-flat complete minimal surface in $\mathbb R^3$ that is conformally an annulus $\{ z  : 0 < 1/r < |z| < r \}$ must also assume every value, with at most 4 exceptions. In 2007, L. Jin and M. Ru (\cite{JR}) extended Kao's result to minimal surfaces in $\mathbb R^m.$ They  
proved : \\

\noindent {\bf Theorem B. }\ {\it
Let $M$ be a non-flat complete minimal surface in $\mathbb R^m$ and let $A$ be an
annular end of $M$ which is conformal to $\{z  :  0 < 1/r < |z| < r\}$, where z
is a conformal coordinate. Then the restriction to $A$ of the (generalized) Gauss map of $M$ can not omit more than $ m(m+1)/2$ hyperplanes in general position in $ \mathbb P^{m-1}(\mathbb C).$ }

Recently, the two first named authors (\cite{DH}) gave an  improvement of the Theorem of Kao. Moreover they also gave an analogue result for the case $m=4.$ 
In this paper we will consider the corresponding problem for the (generalized) Gauss map for non-flat complete minimal surfaces in $\mathbb R^m$ for all $m \geq 3$. 
In this general situation we obtain the following :\\

\noindent {\bf Main Theorem.} \
{\it Let $M$ be a non-flat complete minimal surface in $\mathbb R^m$ and let $A$ be an
annular end of $M$ which is conformal to $\{z  :  0 < 1/r < |z| < r\}$, where z
is a conformal coordinate. Assume that the generalized Gauss map $g$ of $M$ is $k-$non-degenerate on $A$ (that is $g(A)$ is contained in a $k-$dimensional linear subspace in $ \mathbb P^{m-1}(\mathbb C)$, but none of lower dimension), $1\leq k \leq m-1$. If there are $q$ 
hyperplanes  $\{H_j\}_{j=1}^q$ in $N$-subgeneral position in $ \mathbb P^{m-1}(\mathbb C)$ $(N \geq m-1)$ such that $g$ is ramified over
$H_j$ with multiplicity at least $m_j$ on $A$ for each j, then 
\begin{equation}\label{1}
\sum_{j=1}^q(1 - \frac{k}{m_j})\leq (k+1)(N-\dfrac{k}{2})+(N+1). \   
\end{equation}
Moreover, (\ref{1})  still holds if we replace, for all $j=1,...,q$,  $m_j$ by the limit inferior
of the orders of the zeros of the function $ (g, H_j):= \overline{c}_{j0}g_1 +\cdots + \overline{c}_{jm-1}g_{m-1}$ on $A$ (where $g =(g_0: \cdots : g_{m-1})$ is a reduced representation and, for all  $ 1 \leq j \leq q $, the hyperplane $H_j$ in $\mathbb P^{m-1}(\mathbb C)$ is given by $H_j: \overline{c}_{j0}\omega_0+\cdots+ \overline{c}_{jm-1}\omega_{m-1}=0,$ where we assume that $\sum_{i=0}^{m-1}|c_{ji}|^2 = 1$) or by $\infty$ if $g$ intersects $H_j$ only a finite number of times on $A$.}\\

\noindent {\bf Corollary 1.}\ 
{\it Let $M$ be a non-flat complete minimal surface in $\mathbb R^m$ and let $A$ be an
annular end of $M$ which is conformal to $\{z  :  0 < 1/r < |z| < r\}$, where z
is a conformal coordinate. If there are $q$
 hyperplanes  $\{H_j\}_{j=1}^q$ in $N$-subgeneral position in $ \mathbb P^{m-1}(\mathbb C)$ $(N \geq m-1)$ such that the generalized Gauss map $g$ of $M$ is ramified over
$H_j$ with multiplicity at least $m_j$ on $A$ for each j, then 
\begin{equation}\label{2}
\sum_{j=1}^q(1 - \frac{m-1}{m_j})\leq m(N-\dfrac{m-1}{2})+(N+1). \  
\end{equation}
In particular if the hyperplanes  $\{H_j\}_{j=1}^q$ are in general position in $ \mathbb P^{m-1}(\mathbb C)$ we have 
\begin{equation}\label{2'}
\sum_{j=1}^q(1 - \frac{m-1}{m_j})\leq \dfrac{m(m+1)}{2}. \   
\end{equation}
Moreover, (\ref{2}) and (\ref{2'})  still hold if we replace, for all $j=1,...,q$,  $m_j$ by the limit inferior
of the orders of the zeros of the function $ (g, H_j):= \overline{c}_{j0}g_1 +\cdots + \overline{c}_{jm-1}g_{m-1}$ on $A$ (where $g =(g_0: \cdots : g_{m-1})$ is a reduced representation and, for all  $ 1 \leq j \leq q $, the hyperplane $H_j$ in $\mathbb P^{m-1}(\mathbb C)$ is given by $H_j: \overline{c}_{j0}\omega_0+\cdots+ \overline{c}_{jm-1}\omega_{m-1}=0,$ where we assume that $\sum_{i=0}^{m-1}|c_{ji}|^2 = 1$) or by $\infty$ if $g$ intersects $H_j$ only a finite number of times on $A$.
}\\

Our Corollary 1 gives the following improvement of Theorem B of Jin-Ru :\\

\noindent {\bf Corollary 2.}\ {\it If the (generalized) Gauss map $g$ on an annular end 
of a non-flat complete minimal surface in $\mathbb R^m$
assumes $m(m+1)/2$ hyperplanes in general position only finitely often, it takes any other hyperplane in general position (with respect to the previous hyperplanes) infinitely often  with ramification at most $m-1$.}\\

{\bf   Remark.}\  It is well known that the image of the (generalized) Gauss map
$g: M \rightarrow  \mathbb P^{m-1}$ is contained in the hyperquadric $Q_{m-2} \subset \mathbb P^{m-1}$,
and that $Q_1(\mathbb C)$ is biholomorphic to $\mathbb P^1(\mathbb C)$ and that $Q_2(\mathbb C)$ is biholomorphic to $\mathbb P^1(\mathbb C)\times\mathbb P^1(\mathbb C) $. So the results in Dethloff-Ha (\cite{DH}) which only treat the cases $m=3$ and $m=4$ are better than a result which holds for any $m \geq 3$ can be 
if restricted to the special cases $m=3,4$. The easiest way to see the difference is
to observe that $6$ lines in $\mathbb P^2$ in general position may have only $4$
points of intersection with the quadric $Q_1 \subset \mathbb P^2$.\\

The main idea to prove the Main Theorem is to construct and to compare explicit singular flat and negatively curved complete metrics with ramification on these annular ends. This generalizes previous work of Dethloff-Ha (\cite{DH})
(which itself was a refinement of ideas of Ru (\cite{Ru2})) to targets of higher dimensions, which needs among others to combine these explicit singular
metrics with the use of technics from hyperplanes in subgeneral position and 
with the use of intermediate contact functions.  After that we use arguments similar to those used by Kao (\cite{Kao}) and Fujimoto (\cite{Fu} - \cite{Fu3}) to finish the proofs.

\section{Preliminaries}
Let $f$ be a linearly non-degenerate holomorphic map of $\Delta_R :=\{z \in \mathbb C :  |z| < R\}$ into $\mathbb P^k(\mathbb C),$ where $0 < R \leq +\infty.$
 Take a reduced representation $f = (f_0: \cdots : f_k)$. Then $F := (f_0, \cdots, f_k): \Delta_R \rightarrow \mathbb C^{k+1} \setminus \{0\}$ is a holomorphic map with $\mathbb P(F) = f.$ 
 Consider the holomorphic map
\begin{equation*}
F_p=(F_p)_{z}:=F^{(0)}\wedge F^{(1)}\wedge\cdots\wedge F^{(p)}:\Delta_R\longrightarrow \wedge^{p+1}\mathbb{C}^{k+1}
\end{equation*}
for $0\le p \le k,$  where  $F^{(0)} := F= (f_0,\cdots,f_k)$ and $F^{(l)}=(F^{(l)})_{z}:=(f_0^{(l)},\cdots,f_k^{(l)})$ for each $l=0,1,\cdots,k$, and where the $l$-th derivatives 
$f_i^{(l)}=(f_i^{(l)})_{z}$, $i=0,...,k$, are taken with respect to $z$.
(Here and for the rest of this paper the index $|_{z}$ means that the corresponding term
is defined by using differentiation with respect to the variable $z$, and in order to keep notations simple, we usually drop this index if no confusion is possible.)
The norm of $F_p$ is given by
\begin{equation*}
\left| F_p\right|:= \bigg(\sum_{0\leq i_0<\cdots<i_p\leq k}\left|W(f_{i_0},\cdots,f_{i_p})\right|^2\bigg)^\frac{1}{2},
\end{equation*}
where $W(f_{i_0},\cdots,f_{i_p}) = W_z(f_{i_0},\cdots,f_{i_p})$ denotes the Wronskian of $f_{i_0},\cdots,f_{i_p}$ with respect to $z$.
\begin{proposition}(\cite[Proposition 2.1.6]{Fu3})\label{W}.\\
For two holomorphic local coordinates $z$ and $\xi$ and a holomorphic function 
$h : \Delta_R \rightarrow \mathbb{C}$, the following holds :\\
a) $  W_{\xi}(f_0,\cdots, f_p)= W_z(f_0,\cdots, f_p) \cdot (\frac{dz}{d\xi})^{p(p+1)/2}$.\\
b) $W_{z}(hf_0,\cdots, hf_p)= W_z(f_0,\cdots, f_p) \cdot (h)^{p+1}. $ 
\end{proposition}
\begin{proposition}(\cite[Proposition 2.1.7]{Fu3})\label{W1}.\\
For holomorphic functions $f_0, \cdots , f_p : \Delta_R \rightarrow \mathbb{C}$ the following conditions are equivalent:\\ 
(i)\ $ f_0, \cdots, f_p$ are linearly dependent over $\mathbb C.$\\
(ii)\ $W_z(f_0,\cdots,f_p) \equiv 0$ for some (or all) holomorphic local coordinate $z.$
\end{proposition}

We now take a hyperplane $H$ in $\mathbb P^k(\mathbb C)$ given by
\begin{equation*}
H:\overline{c}_0\omega_0+\cdots+\overline{c}_k\omega_k=0\,,
\end{equation*}
with $\sum_{i=0}^k|c_i|^2 = 1.$
We set 
\begin{equation*}
F_0(H) :=F(H):=\overline{c}_0f_0+\cdots+\overline{c}_kf_k
\end{equation*}
 and 
\begin{equation*}
\left| F_p(H)\right|=\left| (F_p)_{z}(H)\right|:= \bigg(\sum_{0\leq i_1< \cdots<i_p\leq k}\left|\sum_{l\not= i_1,...,i_p}\overline{c}_lW(f_{l},f_{i_1}, \cdots,f_{i_p})\right|^2\bigg)^\frac{1}{2},
\end{equation*}
for $1\leq p \leq k.$ We note that by using Proposition \ref{W}, $\left| (F_p)_{z}(H)\right|$ is multiplied by a factor $|\frac{dz}{d\xi}|^{p(p+1)/2}$ if we choose another holomorphic local coordinate $\xi$, and it is multiplied by $|h|^{p+1}$ if we choose another reduced representation $f=(hf_0:\cdots:hf_k)$ with a nowhere zero holomorphic function $h.$
Finally, for $0 \leq p \leq k$, set the $p$-th contact function of $f$ for $H$ to be $\phi_p(H):=\dfrac{|F_p(H)|^2}{|F_p|^2}=\dfrac{|(F_p)_{z}(H)|^2}{|(F_p)_{z}|^2}$.\\

We next consider $q$ hyperplanes $H_1,\cdots,H_q$ in $\mathbb{P}^{k}(\mathbb{C})$ given by
 $$H_j:\left\langle \omega ,A_j\right\rangle \equiv \overline{c}_{j0}\omega_0+\cdots+\overline{c}_{jk}\omega_k \quad(1\leq j\leq q)$$
where $A_j:=(c_{j0},\cdots,c_{jk})$ with $\sum_{i=0}^k|c_{ji}|^2 = 1.$

Assume now $N \geq k$ and $q \geq N+1$.
For $R\subseteq Q:=\left\{1,2,\cdots,q\right\},$ denote by $d(R)$ the dimension of the vector subspace of $\mathbb C^{k+1}$ generated by $\left\{A_j;j\in R \right\}$.\\
\indent The hyperplanes $H_1,\cdots,H_q$ are said to be in $N$-subgeneral position if $d(R)=k+1$ for all $R\subseteq Q$ with $\sharp (R)\geq N+1,$ where $\sharp (A)$ means the number of elements of a set $A.$ In the particular case $N=k$, these are said to be in general position.
\begin{theorem}(\cite[Theorem 2.4.11]{Fu3}) \label{T1}
\emph {\it For given hyperplanes $H_1,\cdots,H_q $ $( q > 2N - k + 1)$ in $\mathbb{P}^k(\mathbb{C})$ located in $N$-subgeneral position, there are some rational numbers $\omega(1),\cdots,\omega(q)$ and $\theta$ satisfying the following conditions:\\
 \indent (i) $0<\omega(j)\leq \theta\leq 1  \quad(1\leq j\leq q),$\\
 \indent (ii) $ \sum^{q}_{j=1}\omega(j)=k+1+\theta(q-2N+k-1),$\\
 \indent (iii) $\frac{k+1}{2N-k+1}\leq \theta \leq \frac{k+1}{N+1},$\\
 \indent (iv) If $R\subset Q$ and $0< \sharp(R)\leq n+1,$ then $\sum_{j \in R} \omega(j)\leq d(R).$}
 \end{theorem}
\noindent  Constants $\omega(j)\ (1 \leq  j  \leq q)$ and $\theta$ with the properties of Theorem~\ref{T1} are called Nochka weights and a Nochka constant for $H_1, \cdots, H_q$ respectively.  Related to Nochka weights, we have the following.
\begin{proposition}(\cite[Proposition 2.4.15]{Fu3})
\label{P1}
Let $H_1,\cdots,H_q$ be hyperplanes in $\mathbb{P}^k(\mathbb{C})$ located in $N$-subgeneral position and let $\omega(1),\cdots,\omega(q)$ be Nochka weights for them, where $q>2N-k+1$. For each $R\subseteq Q:= \left\{1,2,\cdots,q\right\}$ with $0< \sharp (R)\leq N+1$ and real constants $E_1,\cdots,E_q$ with $E_j \geq 1$, there is some $R'\subseteq R$ such that $\sharp (R')=d(R)=d(R')$ and
\begin{equation*}
\prod_{j\in R}E^{\omega(j)}_{j}\leq \prod_{j\in R'}E_j.
\end{equation*}
\end{proposition}

We need the three following results of Fujimoto combining the previously introduced concept of contact functions with Nochka weights:

\begin{theorem}(\cite[Theorem 2.5.3]{Fu3}) \label{PL1}
Let $H_1, \cdots, H_q$ be hyperplanes in $\mathbb P^k(\mathbb C)$ located in $N-$subgeneral position and let $\omega(j)$ $(1\leq j \leq q)$ and $\theta$ be Nochka weights and a Nochka constant for these hyperplanes. For every $\epsilon > 0$ there exist some positive numbers $\delta(>1)$ and $C,$ depending only on $\epsilon$ and $H_j\,$, $1\leq j \leq q,$ such that
\begin{align}\label{eq:T1}
&dd^c\log \dfrac{\Pi_{p=0}^{k-1}|F_p|^{2\epsilon}}{\Pi_{1\leq j\leq q, 0\leq p\leq k-1}\log^{2\omega(j)}(\delta/\phi_p(H_j))}\nonumber \\
&\geq C\bigg( \dfrac{|F_0|^{2\theta(q-2N+k-1)}|F_k|^2}{\Pi_{j=1}^q(|F(H_j)|^2\Pi_{p=0}^{k-1}\log^2(\delta/\phi_p(H_j)))^{\omega(j)}}\bigg)^{\frac{2}{k(k+1)}}dd^c|z|^2.
\end{align}
\end{theorem}

\begin{proposition}(\cite[Proposition 2.5.7]{Fu3})\label{P} \
 Set $\sigma_p=p(p+1)/2$ for $0 \leq p \leq k$ and $\tau_k = \sum_{p=0}^k\sigma_p.$ Then,
\begin{align}
dd^c\log(|F_0|^2|F_1|^2\cdots|F_{k-1}|^2)\geq \dfrac{\tau_k}{\sigma_k}\bigg(\dfrac{|F_0|^2|F_1|^2\cdots|F_{k}|^2}{|F_0|^{2\sigma_{k+1}}}\bigg)^{1/\tau_k}dd^c|z|^2.
\end{align}
\end{proposition}
\begin{proposition} (\cite[Lemma 3.2.13]{Fu3})
\label{P7}
 Let $f$ be a non-degenerate holomorphic map of a domain in  $\mathbb{C}$ into $\mathbb{P}^k(\mathbb{C})$ with reduced representation $f=(f_0:\cdots :f_k)$ and let $H_1,\cdots,H_q$ be hyperplanes located in $N$-subgeneral position $( q > 2N - k + 1)$ with Nochka weights $\omega(1),\cdots,\omega(q)$ respectively. Then,
$$\nu_\phi + \sum^{q}_{j=1}\omega(j)\cdot\min (\nu_{(f, H_j)}, k ) \geq 0,$$
where $\phi = \dfrac{|F_k|}{\Pi_{j=1}^q \mid F(H_j)\mid^{\omega (j)}}.$
\end{proposition}
\begin{lemma}(Generalized Schwarz's Lemma \cite{Ah}). \label{L3}
Let $v$ be a non-negative real-valued continuous subharmonic function on $\Delta_R.$ If $v$ satisfies the inequality  $\Delta\log v \geq v^2$ in the sense of distribution, then
$$v(z) \leq \dfrac{2R}{R^2 - |z|^2}.$$
\end{lemma}

\begin{lemma}\label{ML}
\ Let $f= ( f_0: \cdots : f_k ) : \Delta_R \rightarrow \mathbb P^k (\mathbb{C})$ be a non-degenerate holomorphic map, $ H_1, ... , H_q$ be hyperplanes in $ \mathbb P^k (\mathbb{C}) $ in $N-$subgeneral position ($N \geq k$ and $q > 2N-k+1$), and $\omega(j)$ be their Nochka weights. If  $$\gamma := \sum_{j=1}^{q}\omega(j)(1- \dfrac{k}{m_j}) - (k+1) > 0$$ and $f$ is ramified over $H_j$ with multiplicity at least $m_j \geq k$ for each $j, (1\leq j \leq q),$ then 
for any positive $\epsilon$ with 
$\gamma  > \epsilon \sigma_{k+1}$
there exists a positive constant  $C,$ depending only on $\epsilon, H_j,m_j,\omega(j) (1\leq j \leq q),$ such that
\begin{equation*}
|F|^{\gamma - \epsilon\sigma_{k+1}}\dfrac{|F_k|^{1+\epsilon}\prod_{j=1}^{q}\prod_{p=0}^{k-1}|F_p(H_j)|^{\epsilon /q}}{\prod_{j=1}^{q}|F(H_j)|^{\omega (j)(1-\frac{k}{m_j})}}\leqslant C(\dfrac{2R}{R^2 -|z|^2})^{\sigma_k +\epsilon\tau_k}.
\end{equation*}
\end{lemma}

\begin{proof}
For an arbitrary holomorphic local coordinate $z$ and $\delta(>1) $ chosen as in Theorem~\ref{PL1} we set
\begin{equation*}
\eta_z := \bigg( \dfrac{|F|^{\gamma -\epsilon\sigma_{k+1}}.|F_k|.\prod_{p=0}^{k}|F_p|^{\epsilon}}{\prod_{j=1}^{q}(|F(H_j)|^{(1-\frac{k}{m_j})}\Pi_{p=0}^{k-1}\log(\delta/\phi_p(H_j)))^{\omega (j)}}\bigg)^{\frac{1}{\sigma_k +\epsilon\tau_k}},
\end{equation*}
and define the pseudometric $d\tau_z^2 := \eta_z^2|dz|^2.$ Using Proposition \ref{W} we can see that 
\begin{align*}
d\tau_{\xi} &:= \bigg( \dfrac{|F|^{\gamma -\epsilon\sigma_{k+1}}.|(F_k)_{\xi}|.\prod_{p=0}^{k}|(F_p)_{\xi}|^{\epsilon}}{\prod_{j=1}^{q}(|F(H_j)|^{(1-\frac{k}{m_j})}\Pi_{p=0}^{k-1}\log(\delta/\phi_p(H_j)))^{\omega (j)}}\bigg)^{\frac{1}{\sigma_k +\epsilon\tau_k}}|d\xi|\\
&= \bigg( \dfrac{|F|^{\gamma -\epsilon\sigma_{k+1}}.|(F_k)_{z}||\frac{dz}{d\xi}|^{\sigma_k}.\prod_{p=0}^{k}|(F_p)_{z}|^{\epsilon}.|\frac{dz}{d\xi}|^{\sum_{j=0}^{k}\epsilon\frac{j(j+1)}{2}}}{\prod_{j=1}^{q}(|F(H_j)|^{(1-\frac{k}{m_j})}\Pi_{p=0}^{k-1}\log(\delta/\phi_p(H_j)))^{\omega (j)}}\bigg)^{\frac{1}{\sigma_k +\epsilon\tau_k}}|\dfrac{d\xi}{dz}|.|dz|\\
&= \bigg( \dfrac{|F|^{\gamma -\epsilon\sigma_{k+1}}.|(F_k)_{z}|.\prod_{p=0}^{k}|(F_p)_{z}|^{\epsilon}.|\frac{dz}{d\xi}|^{\sigma_k +\epsilon\tau_k}}{\prod_{j=1}^{q}(|F(H_j)|^{(1-\frac{k}{m_j})}\Pi_{p=0}^{k-1}\log(\delta/\phi_p(H_j)))^{\omega (j)}}\bigg)^{\frac{1}{\sigma_k +\epsilon\tau_k}}|\dfrac{d\xi}{dz}|.|dz|\\
&=d\tau_{z}.
\end{align*}
Thus $d\tau_z^2$ is independent of the choice of the local coordinate $z.$ We will denote $d\tau_z^2$ by $d\tau^2$ for convenience.

We now show that $d\tau$ is continuous on $\Delta_R.$ Indeed, it is easy to see that $d\tau$ is continuous at every point $z_0$ with $\Pi_{j=1}^qF(H_j)(z_0) \not=0.$ Now we take a point $z_0 $ such that $\Pi_{j=1}^qF(H_j)(z_0) =0.$ We have
\begin{align*}
\nu_{d\tau}(z_0)& \geq \dfrac{1}{\sigma_k + \epsilon\tau_k}\bigg( \nu_{F_k}(z_0)-\sum_{j=1}^q\omega(j)\nu_{F(H_j)}(z_0)(1-\dfrac{k}{m_j}) \bigg) \\ 
&=\dfrac{1}{\sigma_k + \epsilon\tau_k}\bigg(\nu_{F_k}(z_0)-\sum_{j=1}^q\omega(j)\nu_{F(H_j)}(z_0)+\sum_{j=1}^q\omega(j)\dfrac{k}{m_j}\nu_{F(H_j)}(z_0) \bigg).
\end{align*}
Combining this with Proposition \ref{P7} we get
\begin{align*}
\nu_{d\tau}(z_0)&\geq \dfrac{1}{\sigma_k + \epsilon\tau_k}\bigg(-\sum_{j=1}^q\omega(j)\min\{\nu_{F(H_j)}(z_0), k \}+\sum_{j=1}^q\omega(j)\dfrac{k}{m_j}\nu_{F(H_j)}(z_0) \bigg).
\end{align*}
By assumption, it holds that $\nu_{F(H_j)}(z_0) \geq m_j \geq k$ or $\nu_{F(H_j)}(z_0)=0,$ so $\nu_{d\tau}(z_0) \geq 0.$ This concludes the proof that $d\tau$ is continuous on $\Delta_R.$

Using Proposition \ref{P}, Theorem \ref{PL1} and noting that $dd^c\log|F_k| = 0$, we have
\begin{align*}
dd^c\log\eta_z&=\dfrac{\gamma - \epsilon\sigma_{k+1}}{\sigma_{k}+\epsilon\tau_k}dd^c\log|F|+\dfrac{\epsilon}{4(\sigma_{k}+\epsilon\tau_k)}dd^c\log(|F_0|^2\cdots|F_{k-1}|^2)\\ 
& + \dfrac{1}{2(\sigma_k +\epsilon\tau_k)}dd^c\log\dfrac{\prod_{p=0}^{k-1}|F_p|^{2(\frac{\epsilon}{2})}}{\prod_{j=1}^{q}\Pi_{p=0}^{k-1}\log^{2\omega (j)}(\delta/\phi_p(H_j))}\\
&\geq \dfrac{\epsilon}{4(\sigma_{k}+\epsilon\tau_k)}\dfrac{\tau_k}{\sigma_k}\bigg(\dfrac{|F_0|^2|F_1|^2\cdots|F_{k}|^2}{|F_0|^{2\sigma_{k+1}}}\bigg)^{1/\tau_k}dd^c|z|^2\\ 
& + C_0\bigg( \dfrac{|F_0|^{2\theta(q-2N+k-1)}|F_k|^2}{\Pi_{j=1}^q(|F(H_j)|^2\Pi_{p=0}^{k-1}\log^2(\delta/\phi_p(H_j)))^{\omega(j)}}\bigg)^{\frac{2}{k(k+1)}}dd^c|z|^2\\
&\geq \min \{\dfrac{1}{4\sigma_k(\sigma_{k}+\epsilon\tau_k)}, \dfrac{C_0}{\sigma_k} \}\bigg(\epsilon\tau_k\bigg(\dfrac{|F_0|^2|F_1|^2\cdots|F_{k}|^2}{|F_0|^{2\sigma_{k+1}}}\bigg)^{1/\tau_k}\\
&+ \sigma_k\bigg( \dfrac{|F_0|^{2\theta(q-2N+k-1)}|F_k|^2}{\Pi_{j=1}^q(|F(H_j)|^2\Pi_{p=0}^{k-1}\log^2(\delta/\phi_p(H_j)))^{\omega(j)}}\bigg)^{\frac{1}{\sigma_k}}\bigg) dd^c|z|^2
\end{align*}
where $C_0$ is the positive constant. So,  by using the basic inequality 
$$ \alpha A + \beta B \geq (\alpha + \beta) A^{\frac{\alpha}{\alpha+ \beta}}B^{\frac{\beta}{\alpha + \beta}} \text{ for all }\alpha, \beta, A, B > 0  ,$$ we can find a positive constant $C_1$ satisfing the following\\
\begin{align*}
dd^c\log\eta_z &\geq C_1\bigg( \dfrac{|F|^{\theta (q-2N+k-1) -\epsilon\sigma_{k+1}}.|F_k|.\prod_{p=0}^{k}|F_p|^{\epsilon}}{\prod_{j=1}^{q}(|F(H_j)|\cdot \Pi_{p=0}^{k-1}\log(\delta/\phi_p(H_j)))^{\omega (j)}}\bigg)^{\frac{2}{\sigma_k +\epsilon\tau_k}}dd^c|z|^2\\
&= C_1\bigg( \dfrac{|F|^{\sum_{j=1}^q \omega (j) - k-1 -\epsilon\sigma_{k+1}}.|F_k|.\prod_{p=0}^{k}|F_p|^{\epsilon}}{\prod_{j=1}^{q}(|F(H_j)|\cdot \Pi_{p=0}^{k-1}\log(\delta/\phi_p(H_j)))^{\omega (j)}}\bigg)^{\frac{2}{\sigma_k +\epsilon\tau_k}}dd^c|z|^2 \  \text{ (by Theorem \ref{T1}) }\\
&= C_1\bigg( \dfrac{|F|^{\gamma -\epsilon\sigma_{k+1}}.|F_k|.\prod_{p=0}^{k}|F_p|^{\epsilon}\prod_{j=1}^{q}\bigg(\dfrac{|F|}{|F(H_j)|}\bigg)^{\frac{k}{m_j}\omega (j)}}{\prod_{j=1}^{q}(|F(H_j)|^{(1-\frac{k}{m_j})}\cdot \Pi_{p=0}^{k-1}\log(\delta/\phi_p(H_j)))^{\omega (j)}}\bigg)^{\frac{2}{\sigma_k +\epsilon\tau_k}}dd^c|z|^2.
\end{align*}
On the other hand, 
\begin{equation*}
\bigg(\dfrac{|F(H_j)|}{|F|}\bigg)^{\frac{k}{m_j}\omega (j)} \leq 1 \text{ for all } j = 1, 2, ..., q,
\end{equation*}
so we get
\begin{align*}
dd^c\log\eta_z &\geq C_1\bigg( \dfrac{|F|^{\gamma -\epsilon\sigma_{k+1}}.|F_k|.\prod_{p=0}^{k}|F_p|^{\epsilon}}{\prod_{j=1}^{q}(|F(H_j)|^{(1-\frac{k}{m_j})}\cdot \Pi_{p=0}^{k-1}\log(\delta/\phi_p(H_j)))^{\omega (j)}}\bigg)^{\frac{2}{\sigma_k +\epsilon\tau_k}}dd^c|z|^2\\
&=C_1\eta_z^2dd^c|z|^2.
\end{align*}
We now use Lemma \ref{L3} to show the following
\begin{equation*}
\bigg( \dfrac{|F|^{\gamma -\epsilon\sigma_{k+1}}.|F_k|.\prod_{p=0}^{k}|F_p|^{\epsilon}}{\prod_{j=1}^{q}(|F(H_j)|^{(1-\frac{k}{m_j})}\cdot \Pi_{p=0}^{k-1}\log(\delta/\phi_p(H_j)))^{\omega (j)}}\bigg)^{\frac{1}{\sigma_k +\epsilon\tau_k}} \leq C_2\dfrac{2R}{R^2 - |z|^2}.
\end{equation*}
We then have
\begin{align*}
&\bigg( \dfrac{|F|^{\gamma -\epsilon\sigma_{k+1}}.|F_k|^{1+\epsilon}.\prod_{j=1}^q\prod_{p=0}^{k-1}|F_p(H_j)|^{\epsilon/q}}{\prod_{j=1}^{q}|F(H_j)|^{(1-\frac{k}{m_j})\omega (j)}\cdot \prod_{j=1}^q \prod_{p=0}^{k-1}\bigg((\dfrac{|F_p(H_j)|}{|F_p|})^{\epsilon/q}\log^{\omega (j)}(\delta/\phi_p(H_j))\bigg)}\bigg)^{\frac{1}{\sigma_k +\epsilon\tau_k}}\\
& \leq C_2\dfrac{2R}{R^2 - |z|^2}.
\end{align*}
Moreover, combining with
\begin{equation*}
\sup_{0<x\leq 1} x^{\frac{\epsilon}{q}}\log^{\omega(j)}\left(\frac{\delta}{x^2}\right)< +\infty,
\end{equation*}
then we get
\begin{align*}
\bigg( \dfrac{|F|^{\gamma -\epsilon\sigma_{k+1}}.|F_k|^{1+\epsilon}.\prod_{j=1}^q\prod_{p=0}^{k-1}|F_p(H_j)|^{\epsilon/q}}{\prod_{j=1}^{q}|F(H_j)|^{(1-\frac{k}{m_j})\omega (j)}}\bigg)^{\frac{1}{\sigma_k +\epsilon\tau_k}}\leq C\dfrac{2R}{R^2 - |z|^2}
\end{align*}
where $C$ is the positive constant
depending, by Theorem~\ref{PL1} and by our construction, only on $\epsilon, H_j,m_j,\omega(j) (1\leq j \leq q)$. This implies Lemma~\ref{ML}.
\end{proof}
We finally will need the following result on completeness of open Riemann surfaces with conformally flat metrics due to Fujimoto :
\begin{lemma}(\cite[Lemma 1.6.7]{Fu3}). \label{L5}
Let $d\sigma^2$ be a conformal flat metric on an open Riemann surface $M$. Then for every point $p \in M$, there is a holomorphic and locally biholomorphic map $\Phi$ of
a disk (possibly with radius $\infty$)  $\Delta_{R_0} := \{w : |w|<R_0 \}$ $(0<R_0 \leq \infty )$ onto an open neighborhood of $p$ with $\Phi (0) = p$ such that $\Phi$ is a local isometry, namely the pull-back 
$\Phi^*(d\sigma^2)$ is equal to the standard (flat) metric on $\Delta_{R_0}$, and for some point $a_0$ with $|a_0|=1$, the $\Phi$-image of the curve 
$$L_{a_0} : w:= a_0 \cdot s \; (0 \leq s < R_0)$$
is divergent in $M$ (i.e. for any compact set $K \subset M$, there exists an $s_0<R_0$
such that the $\Phi$-image of the curve $L_{a_0} : w:= a_0 \cdot s \; (s_0 \leq s < R_0)$
does not intersect $K$).
\end{lemma}

\section{The proof of the Main Theorem}
\begin{proof}
\indent For the convenience of the reader, we first recall some notations on the Gauss map of minimal surfaces in $\mathbb R^m.$
Let $M$ be a complete immersed minimal surface in $\mathbb R^m.$ Take an immersion $x
=(x_0,...,x_{m-1}) : M \rightarrow \mathbb R^m.$ Then $M$ has the structure of a Riemann surface and any local isothermal coordinate $(x, y)$ of $M$ gives a local holomorphic coordinate $z=x+ \sqrt{-1}y$. The generalized Gauss map of $x$ is defined to be 
\begin{equation*}
g : M \rightarrow \mathbb P^{m-1}(\mathbb C), g = \mathbb P(\dfrac{\partial x}{\partial z})=(\dfrac{\partial x_0}{\partial z}:\cdots:\dfrac{\partial x_{m-1}}{\partial z}).
\end{equation*}
Since $x:M \rightarrow \mathbb R^m$ is immersed, $$G=G_{z}:= 
(g_0,...,g_{m-1}) = ((g_0)_{z},...,(g_{m-1})_{z}) =(\dfrac{\partial x_0}{\partial z},\cdots,\dfrac{\partial x_{m-1}}{\partial z})$$ is a (local) reduced representation of $g$, and
since for another local holomorphic coordinate $\xi$ on $M$ we have $G_{\xi} = G_{z}\cdot (\dfrac{dz}{d\xi})$, $g$ is well defined (independently of the (local) holomorphic coordinate). 
Moreover, if $ds^2$ is the metric on $M$ induced by the standard metric on $\mathbb R^m$, we have 
\begin{equation}\label{eq:1}
ds^2 = 2|G_{z}|^2|dz|^2 .
\end{equation}
Finally since $M$ is minimal,  $g$ is a holomorphic map. 

Since by hypothesis of the Main Theorem, $g$ is $k$-non-degenerate $(1 \leq k \leq m-1)$
 without loss of
generality, we may assume that $g(M) \subset \mathbb P^k(\mathbb C);$ then
\begin{equation*}
g : M \rightarrow \mathbb P^{k}(\mathbb C), g = \mathbb P(\dfrac{\partial x}{\partial z})=(\dfrac{\partial x_0}{\partial z}:\cdots:\dfrac{\partial x_{k}}{\partial z}).
\end{equation*}
is linearly non-degenerate in $\mathbb P^k(\mathbb C)$ (so in particular $g$ is not constant)  and the other facts mentioned above still hold.

Let $H_j(j = 1,...,q)$ be $q (\geq N+1)$ hyperplanes in $\mathbb P^{m-1}(\mathbb C)$ in $N$-subgeneral position $(N\geq m-1 \geq k)$. Then  $H_j\cap \mathbb P^{k}(\mathbb C)(j = 1,...,q)$ are $q$ hyperplanes in $\mathbb P^{k}(\mathbb C)$ in $N$-subgeneral position. Let each $H_j\cap \mathbb P^{k}(\mathbb C)$ be represented  as
\begin{equation*}
H_j\cap \mathbb P^{k}(\mathbb C) :\overline{c}_{j0}\omega_0+\cdots+\overline{c}_{j{k}}\omega_{k}=0
\end{equation*}
with $\sum_{i=0}^k|c_{ji}|^2 = 1.$\\
Set
\begin{equation*}
G(H_j) = G_{z}(H_j):= \overline{c}_{j0}g_0+\cdots+\overline{c}_{jk}g_{k}.
\end{equation*}

We will now, for each contact function $\phi_p(H_j)$ for each of our hyperplanes $H_j$, choose one of the components of the numerator $|((G_z)_p)_z(H_j)|$ which is not identically zero: More precisely, for each $j,p\ (1\leq j \leq q, 1\leq p \leq k),$ we can choose $i_1,\cdots,i_p$ with $0\leq i_1<\cdots< i_p \leq k$ such that
\begin{equation*}
\psi(G)_{jp}=(\psi (G_z)_{jp})_{z}:=\sum_{l\neq i_1,..,i_p}\overline{c}_{jl}W_{z}(g_l,g_{i_1},\cdots,g_{i_p})\not\equiv 0,
\end{equation*}
(indeed, otherwise, we have 
 $\sum_{l\neq i_1,..,i_p}\overline{c}_{jl}W(g_l,g_{i_1},\cdots,g_{i_p})\equiv 0$ for all $i_1, ..., i_p$,
so $W(\sum_{l\neq i_1,..,i_p}\overline{c}_{jl}g_l,g_{i_1},\cdots,g_{i_p})\equiv 0$ for all $i_1, ..., i_p$, which contradicts the non-degeneracy of $g$ in $\mathbb P^k(\mathbb C).$
Alternatively we simply can observe that in our situation none of the contact functions vanishes identically.)
We still set $\psi (G)_{j0}=\psi (G_z)_{j0}:=G(H_j)(\not\equiv 0)$, and we also note that $\psi (G)_{jk}=((G_z)_k)_z$.
Since the $\psi (G)_{jp}$ are holomorphic, so they have only isolated zeros. 

Finally we put for later use the transformation formulas for all the terms defined above,
which are obtained by using Proposition \ref{W} :
 For local holomorphic coordinates $z$ and $\xi$ on $M$ we have : 
 \begin{equation}\label{e1}
 G_{\xi} = G_{z}\cdot (\dfrac{dz}{d\xi})\,,
 \end{equation}
  \begin{equation}\label{e2}
 G_{\xi}(H) = G_{z}(H) \cdot (\dfrac{dz}{d\xi})\,,
 \end{equation}
   \begin{equation}\label{e3}
  ((G_{\xi})_k)_{\xi}=((G_z)_k)_{z}\cdot (\dfrac{dz}{d\xi})^{k+1+\frac{k(k+1)}{2}}=((G_z)_k)_{z}(\dfrac{dz}{d\xi})^{\sigma_{k+1}}\,,
  \end{equation}
   \begin{equation}\label{e4}
  (\psi (G_{\xi})_{jp})_{\xi}=(\psi (G_z)_{jp})_{z}\cdot (\dfrac{dz}{d\xi})^{p+1+\frac{p(p+1)}{2}}=(\psi (G_z)_{jp})_{z} \cdot (\dfrac{dz}{d\xi})^{\sigma_{p+1}}\,, \:(0 \leq p \leq k)\,.
  \end{equation} 
Moreover, we also will need the following transformation formulas for mixed variables :
   \begin{equation}\label{e5}
  ((G_{\xi})_k)_{\xi}=((G_{\xi})_k)_{z}\cdot (\dfrac{dz}{d\xi})^{\frac{k(k+1)}{2}}=((G_{\xi})_k)_{z}(\dfrac{dz}{d\xi})^{\sigma_{k}}\,,
  \end{equation}
   \begin{equation}\label{e6}
  (\psi (G_{\xi})_{jp})_{\xi}=(\psi (G_{\xi})_{jp})_{z}\cdot (\dfrac{dz}{d\xi})^{\frac{p(p+1)}{2}}=(\psi (G_{\xi})_{jp})_{z} \cdot (\dfrac{dz}{d\xi})^{\sigma_{p}}\,, \:(0 \leq p \leq k)\,.
  \end{equation}

Now we prove the Main Theorem in four steps:\\

{\bf Step 1:}  
We will fix notations on the annular end $A \subset M$. Moreover, by passing  to  a sub-annular end of $A \subset M$ we simplify the geometry of the Main Theorem.

Let $A \subset M$ be an annular end of $M,$ that is, $A = \{z  :  0 < 1/r < |z| < r < \infty  \},$ where $z$ is a (global) conformal coordinate of $A$. Since $M$ is complete with respect to $ds^2$, we may assume that the restriction of
 $ds^2$ to $A$ is complete on the set  $\{ z : |z| = r\}$, i.e., the set $\{ z : |z| = r\}$ is at infinite distance from any  point  of $A$. 
 
 Let $m_j$ be the  limit inferior
of the orders of the zeros of the functions $G(H_j)$ on $A$, or
 $m_j = \infty$ if $G(H_j)$ has only a finite number of zeros on $A$. 
 
 All the $m_j$ are increasing if we only consider the zeros which the functions $G(H_j)$ take on a subset $B \subset A$. So without loss of generality we may prove our theorem only
 on a sub-annular end, i.e., a subset $ A_t :=\{z  :  0 < t \leq |z| < r < \infty  \} \subset A$ 
 with some $t$ such that $1/r < t<r$. (We trivially observe that for $c:=tr>1$, $s:= r/\sqrt{c}$, $\xi := z/ \sqrt{c}$, we have $A_t = \{\xi  :  0 < 1/s \leq |\xi| < s < \infty  \}$.) 
 
 By passing to such a sub-annular end we will be able to extend the construction of a metric in step 2 below to the set $\{ z : |z|=1/r \}$, and, moreover, we may assume that for all $j=1,...,q$ :
 \begin{equation} \label{ass1}
 g\: {\rm omits}\: H_j\: (m_j=\infty)\: {\rm or} \:{\rm takes}\: H_j \:{\rm infinitely}\: {\rm often}
 \:{\rm with}\:{\rm ramification } \end{equation}
\begin{equation*}
 m_j< \infty \:{\rm and}\: {\rm is}\: 
 {\rm ramified}\: {\rm over}\: H_j\: 
 {\rm with}\: {\rm multiplicity}\: {\rm at}\: {\rm least}\: m_j.
\end{equation*}
We next observe that we may also assume 
 \begin{equation} \label{ass1'}
m_j > k\,,\:j=1,...,q\,.
\end{equation}
In fact, if this does not hold for all $j=1,...,q$, we just drop the $H_j$ for which it does not hold, and remain with $\tilde{q}<q$ such hyperplanes. If $\tilde{q} \geq N+1$, they are still
in $N$-subgeneral position in $\mathbb P^{m-1}(\mathbb C)$ and we prove our Main Theorem for $\tilde{q}$ instead of $q$, if $ \tilde{q} < N+1$, the assertion (\ref{1}) of our Main Theorem trivially holds. In both cases since by passing from $\tilde{q}$ to $q$ again 
the right hand side of (\ref{1}) does not change, however the left hand side only
becomes possibly smaller, the inequality (\ref{1}) still holds if we (re-)consider all the $q$ hyperplanes and we are done.\\

 {\bf Step 2:} On the annular end $A = \{z  :  0 < 1/r \leq |z| < r < \infty  \}$ minus a discrete
 subset $S \subset A$ we construct a flat metric $d\tau^2$ on $A \setminus S$  which
 is complete on the set  $\{ z : |z| = r\} \cup S$, i.e., the set $\{ z : |z| = r\} \cup S $ is at infinite distance from any  point  of $A\setminus S$. We may assume that
\begin{equation}\label{eq:2}
\sum_{j=1}^q(1 - \frac{k}{m_j}) > (k+1)(N-\dfrac{k}{2})+(N+1)\,, \   
\end{equation} 
otherwise our Main Theorem is already proved.
By (\ref{eq:2}), we get 
\begin{equation}\label{eq:2'}
(\sum_{j=1}^{q}(1- \dfrac{k}{m_j}))-2N+k-1 > \dfrac{(2N-k +1)k}{2}>0\,,
\end{equation}
and by (\ref{ass1'}) this implies in particular
 \begin{equation} \label{ass1''}
q>2N-k+1 \geq N+1 \geq k+1\,.
\end{equation}
 By Theorem \ref{T1}, (\ref{ass1''}) and (\ref{eq:2'}), we have \\
$ (q - 2N + k - 1)\theta = (\sum_{j=1}^q \omega (j)) - k - 1\,,$\\ $\  \theta \geq \omega(j) > 0$ 
and $  \theta \geq \dfrac{k + 1}{2N - k + 1},$\\
so 
\begin{align*}
2\bigg ((\sum_{j=1}^q \omega(j)(1-\dfrac{k}{m_j})) - k - 1\bigg )&= \dfrac{2((\sum_{j=1}^{q}\omega(j))- k - 1)\theta}{\theta} - 2\sum_{j=1}^{q}\dfrac{k\omega(j)\theta}{\theta m_j}\\
&= 2(q-2N + k - 1)\theta - 2\sum_{j=1}^{q}\dfrac{k\omega(j)\theta}{\theta m_j}\\
&\geq 2(q-2N + k - 1)\theta - 2\sum_{j=1}^{q}\dfrac{k\theta}{m_j}\\
&= 2\theta\bigg ((\sum_{j=1}^{q}(1- \dfrac{k}{m_j}))-2N+k-1\bigg )\\
&\geq 2\dfrac{(k + 1)\bigg ((\sum_{j=1}^{q}(1- \dfrac{k}{m_j}))-2N+k-1\bigg )}{2N-k+1}.
\end{align*}
Thus, we now can  conclude with (\ref{eq:2'}) that 
\begin{align}\label{eq:3}
&2\bigg ((\sum_{j=1}^q \omega (j)(1-\dfrac{k}{m_j})) - k - 1\bigg ) > k(k+1) \nonumber \\
&\Rightarrow (\sum_{j=1}^q \omega (j)(1-\dfrac{k}{m_j})) - k - 1 - \dfrac{k(k+1)}{2} >0.
\end{align}
By (\ref{eq:3}), we can choose a number $\epsilon(>0) \in \mathbb Q $ such that \\
\begin{align*}
&\dfrac{\sum_{j=1}^q \omega(j)(1-\frac{k}{m_j}) - (k + 1) - \frac{k(k + 1)}{2}}{\tau_{k+1}} >  \epsilon >\\
& > \dfrac{\sum_{j=1}^q \omega (j)(1-\frac{k}{m_j}) - (k + 1) - \frac{k(k + 1)}{2}}{\frac{1}{q} + \tau_{k+1} }.
\end{align*}
So
\begin{equation} \label{eq:5}
h:=(\sum_{j=1}^q \omega(j)(1-\frac{k}{m_j})) - (k + 1) - \epsilon\sigma_{k+1} > \frac{k(k + 1)}{2} + \epsilon\tau_{k}
\end{equation}
and
\begin{equation}\label{eq:6} 
\dfrac{\epsilon}{q}> (\sum_{j=1}^q \omega(j)(1-\frac{k}{m_j})) - (k + 1) - \frac{k(k + 1)}{2} - \epsilon\tau_{k+1}.
\end{equation}
We now consider the number
\begin{equation}\label{nb}
\rho := \dfrac{1}{h} \bigg ( \frac{k(k + 1)}{2} + \epsilon\tau_k\bigg)= \dfrac{1}{h} \bigg ( \sigma_k + \epsilon\tau_k\bigg).
\end{equation}
Then, by (\ref{eq:5}), we have
 \begin{equation}\label{eq:7}
0 < \rho < 1.
\end{equation}
Set
\begin{equation}\label{eq:8}
\rho^* := \dfrac{1}{(1 - \rho)h}=\dfrac{1}{(\sum_{j=1}^q \omega(j)(1-\frac{k}{m_j})) - (k + 1) - \frac{k(k + 1)}{2} - \epsilon\tau_{k+1}}.
\end{equation}
Using (\ref{eq:6}) we get 
\begin{equation}\label{eq:9}
\dfrac{\epsilon\rho^*}{q} > 1.
\end{equation}

Consider the open subset
\begin{equation*}
A_1 = Int(A) -\cup_{j=\overline{1,q},p=\overline{0,k}}\{ z | \psi (G)_{jp} = 0 \}
\end{equation*}
of $A$. Using the global holomorphic coordinate $z$ on $A \supset A_1$ we define a new pseudo metric
\begin{equation}\label{eq:10}
d\tau^2 
= \bigg(\dfrac{\Pi_{j=1}^q|G_{z}(H_j)|^{\omega(j)(1-\frac{k}{m_j})}}{|((G_z)_k)_{z}|^{1+\epsilon}\Pi_{p=0}^{k-1}\Pi_{j=1}^q|(\psi(G_z)_{jp})_{z}|^{\epsilon/q}}\bigg)^{2\rho^*}|dz|^2
\end{equation}
on $A_1.$
We note that by the transformation formulas (\ref{e1}) to (\ref{e4}) 
for a local holomorphic coordinate $\xi$ we have 
\begin{equation}\label{ind1}
\bigg(\dfrac{\Pi_{j=1}^q|G_{z}(H_j)|^{\omega(j)(1-\frac{k}{m_j})}}{|((G_z)_k)_{z}|^{1+\epsilon}\Pi_{p=0}^{k-1}\Pi_{j=1}^q|(\psi (G_z)_{jp})_{z}|^{\epsilon/q}}\bigg)^{2\rho^*}|dz|^2 
\end{equation}
\begin{equation*}
= 
\bigg(\dfrac{\Pi_{j=1}^q|G_{\xi}(H_j)|^{\omega(j)(1-\frac{k}{m_j})}}{|((G_{\xi})_k)_{\xi}|^{1+\epsilon}\Pi_{p=0}^{k-1}\Pi_{j=1}^q|(\psi(G_{\xi})_{jp})_{\xi}|^{\epsilon/q}}\bigg)^{2\rho^*}|d\xi|^2
\end{equation*}
so the pseudo metric $d\tau$ is in fact defined independently of the choice of the coordinate.  
Moreover, it is also easy to see that $d\tau$ is flat. 

Next we observe that for any point $z \in A$, we have
\begin{equation}\label{lemma}
(\nu_{G_k} - \sum_{j=1}^q \omega (j) \nu_{G(H_j)}(1-\frac{k}{m_j}))(z) \geq 0 \,.
\end{equation}
In fact, put $\phi:= \frac{|G_k|}{\prod_{j=1}^q |G(H_j)|^{\omega(j)}}$. 
Observing that by (\ref{ass1'}) for all $j=1,...,q$ and all $z \in A$ we have either 
$\nu_{G(H_j)}(z)=0$ or $\nu_{G(H_j)}(z) \geq m_j \geq k$, we get
$$\frac{k}{m_j}\nu_{G(H_j)} \geq \min \{\nu_{G(H_j)}, k\} \,.
$$ 
So by Lemma \ref{P7} we have
\begin{align*}
&\nu_{G_k} - \sum_{j=1}^q \omega (j) \nu_{G(H_j)}(1-\frac{k}{m_j})\\
&= \nu_{\phi} + \sum_{j=1}^q \omega (j)\frac{k}{m_j} \nu_{G(H_j)}\\
&\geq \nu_{\phi} + \sum_{j=1}^q \omega (j) \min \{\nu_{G(H_j)}, k\} \geq 0\,.\\
\end{align*}

Now it is easy to see that $d\tau$ is continuous and nowhere vanishing on $A_1.$
Indeed, for $z_0 \in A_1$ with $\Pi_{j=1}^qG(H_j)(z_0) \not= 0,$ $d\tau$ is continuous and not vanishing at $z_0.$
Now assume that there exists $z_0 \in A_1$ such that $G(H_i)(z_0)=0$ for some $i.$ But by (\ref{lemma}) and (\ref{ass1'}) we then get that
$\nu_{G_k}(z_0)>0$ which contradicts to $z_0 \in A_1$.\\

The key point is now to prove following claim. 
\begin{claim} \label{Cl1}
 $d\tau$ is complete on the set  $\{ z : |z| = r\}\cup_{j=\overline{1,q},p=\overline{0,k}} \{z :\psi (G)_{jp}(z)=0 \},$ i.e., set $\{ z : |z| = r\}\cup_{j=\overline{1,q},p=\overline{0,k}} \{z : \psi (G)_{jp}(z)=0 \}$ is at infinite distance from any interior point in $A_1.$ 
\end{claim}
First, assume that $\Pi_{p=0}^{k}\Pi_{j=1}^q|\psi (G)_{jp}|(z_0) = 0.$\\
Then using (\ref{lemma}) we get 
\begin{align*}
\nu_{d\tau}(z_0)&= - \bigg(( \nu_{G_k}(z_0) - \sum_{j=1}^q\omega(j)\nu_{G(H_j)}(z_0)(1-\dfrac{k}{m_j})) + (\epsilon \nu_{G_k}(z_0) + \dfrac{\epsilon}{q}\sum_{j=1}^q \sum_{p=0}^{k-1}\nu_{\psi (G)_{jp}}(z_0))\bigg)\rho^*\\ 
&\leq -\epsilon \rho^*\nu_{G_k}(z_0) -\dfrac{\epsilon\rho^*}{q}\sum_{j=1}^q \sum_{p=0}^{k-1}\nu_{\psi (G)_{jp}}(z_0)\leq -\dfrac{\epsilon\rho^*}{q}.
\end{align*}
Thus we can find a positive constant $C$ such that 
\begin{equation*}
|d\tau| \geq \dfrac{C}{ |z-z_0|^{\frac{\epsilon\rho^*}{q}}}|dz|
\end{equation*}
in a neighborhood of $z_0$ and then, combining with (\ref{eq:9}), $d\tau$ is complete on $\cup_{j=\overline{1,q},p=\overline{0,k}} \{z | \psi (G)_{jp}(z)=0 \}.$ \\
\indent Now assume that $d\tau$ is not complete on $\{z: |z| = r \}.$ Then there exists $\gamma: [0, 1) \rightarrow A_1,$ where $\gamma (1) \in \{z : |z| = r \},$ so that $|\gamma| < \infty.$ Furthermore, we may also assume that $dist(\gamma (0); \{z : |z| = 1/r\}) > 2|\gamma|.$ Consider a small disk $\Delta$ with center at $\gamma (0).$ Since $d\tau$ is flat, $\Delta$ is isometric to an ordinary disk in the plane (cf. e.g.  Lemma \ref{L5}). Let $\Phi: \{w: |w| < \eta \}\rightarrow \Delta$ be this isometry. Extend $\Phi,$ as a local isometry into $A_1,$ to the largest disk $\{w: |w| < R\} = \Delta_R$ possible. Then $R \leq |\gamma|.$ The reason that $\Phi$ cannot be extended to a larger disk is that the image goes to the outside boundary $\{z : |z| = r\}$ of $A_1$ (it cannot go to points $z$ of $A$ with $\Pi_{j=\overline{1,q},p=\overline{0,k}} \psi (G)_{jp}(z) =0$ since we have shown already the completeness of $A_1$ with respect 
to these points). More precisely, there exists a point $w_0$ with $|w_0| =R$ so that $\Phi(\overline{0,w_0}) = \Gamma_0$ is a divergent curve on $A.$

Since we want to use Lemma \ref{ML} to finish up step 2, for the rest of the proof of step 2
we consider $G_z=((g_0)_z,...,(g_k)_z)$ as a {\it fixed globally defined reduced representation of $g$} by means of the global coordinate $z$ of $A \supset A_1$.
(We remark that then we loose of course the invariance of $d\tau^2$ under coordinate changes (\ref{ind1}), but since $z$ is a global coordinate this will be no problem 
and we will not need this invariance
for the application of Lemma \ref{ML}.) If again 
$\Phi : \{w: |w| < R\} \rightarrow A_1$ is our maximal local isometry, it is in particular holomorphic and locally biholomorphic. So $f:= g \circ \Phi :  \{w: |w| < R\} \rightarrow
\mathbb P^k (\mathbb C)$ is a linearly non-degenerate holomorphic map with fixed global
reduced representation $$F:= G_z \circ \Phi =((g_0)_z  \circ \Phi,...,(g_k)_z  \circ \Phi)
=(f_0,...,f_k)\,.$$ 
Since $\Phi$ is locally biholomorphic, the metric on $\Delta_R$ induced from $ds^2$ (cf. (\ref{eq:1})) through $\Phi$ is given by
\begin{equation}\label{eq:12}
\Phi^*ds^2 =  2|G_z \circ \Phi|^2|\Phi^* dz|^2 = 
2|F|^2 |\frac{dz}{dw}|^2|dw|^2\,.
\end{equation}
On the other hand, $\Phi$ is locally isometric, so we have
\begin{equation*}
|dw| = |\Phi^*d\tau|= \bigg(\dfrac{\Pi_{j=1}^q|G_z(H_j) \circ \Phi|^{\omega(j)(1-\frac{k}{m_j})}}{|((G_z)_k)_z \circ \Phi|^{1+\epsilon}\Pi_{p=0}^{k-1}\Pi_{j=1}^q|(\psi (G_z)_{jp})_z \circ \Phi|^{\epsilon/q}}\bigg)^{\rho^*}|\frac{dz}{dw}||dw|\,.
\end{equation*}
By (\ref{e5}) and (\ref{e6}) we have
$$
 ((G_z)_k)_z \circ \Phi =((G_z \circ \Phi)_k)_{w}(\dfrac{dw}{dz})^{\sigma_{k}}
 =(F_k)_{w}(\dfrac{dw}{dz})^{\sigma_{k}}\,,$$
 $$
  (\psi (G_{z})_{jp})_{z} \circ \Phi =(\psi (G_{z} \circ \Phi)_{jp})_{w} \cdot (\dfrac{dw}{dz})^{\sigma_{p}}=(\psi (F)_{jp})_{w} \cdot (\dfrac{dw}{dz})^{\sigma_{p}}\,, \:(0 \leq p \leq k)\,\,.
$$
Hence, by definition of $\rho$ in (\ref{nb}), we have
\begin{align*}
|\dfrac{dw}{dz}|&= 
\bigg(\dfrac{\Pi_{j=1}^q|G_z(H_j) \circ \Phi|^{\omega(j)(1-\frac{k}{m_j})}}{|((G_z)_k)_z \circ \Phi|^{1+\epsilon}\Pi_{p=0}^{k-1}\Pi_{j=1}^q|(\psi (G_z)_{jp})_z \circ \Phi|^{\epsilon/q}}\bigg)^{\rho^*}
\\
&= \bigg(\dfrac{\Pi_{j=1}^q|F(H_j)|^{\omega(j)(1-\frac{k}{m_j})}}{|(F_k)_w|^{1+\epsilon}\Pi_{p=0}^{k-1}\Pi_{j=1}^q|(\psi (F)_{jp})_w|^{\epsilon/q}}\bigg)^{\rho^*}\dfrac{1}{|\frac{dw}{dz}|^{h\rho\rho^*}}.
\end{align*}
So by the definition of $\rho^*$ in (\ref{eq:8}), we get
\begin{align*}
|\dfrac{dz}{dw}|&= 
\bigg(
\dfrac{|(F_k)_w|^{1+\epsilon}\Pi_{p=0}^{k-1}\Pi_{j=1}^q|(\psi (F)_{jp})_w|^{\epsilon/q}}
{\Pi_{j=1}^q|F(H_j)|^{\omega(j)(1-\frac{k}{m_j})}}
\bigg)^{\frac{\rho^*}{1+h\rho\rho^*}}\\
&= 
\bigg(
\dfrac{|(F_k)_w|^{1+\epsilon}\Pi_{p=0}^{k-1}\Pi_{j=1}^q|(\psi (F)_{jp})_w|^{\epsilon/q}}
{\Pi_{j=1}^q|F(H_j)|^{\omega(j)(1-\frac{k}{m_j})}}
\bigg)^{\frac{1}{h}} \ .
\end{align*}
Moreover, $|(\psi (F)_{jp})_w| \leq |(F_p)_w(H_j)|$ by the definitions, so we obtain
\begin{equation}\label{eq:13}
|\dfrac{dz}{dw}| \leq  
\bigg(\dfrac{|(F_k)_w|^{1+\epsilon}\Pi_{p=0}^{k-1}\Pi_{j=1}^q|(F_p)_w(H_j)|^{\epsilon/q}}
{\Pi_{j=1}^q|F(H_j)|^{\omega(j)(1-\frac{k}{m_j})}}\bigg)^{\frac{1}{h}} \ .
\end{equation}
By (\ref{eq:12}) and (\ref{eq:13}), we have
\begin{equation*}
\Phi^*ds \leqslant \sqrt{2}|F|
\bigg(\dfrac{|(F_k)_w|^{1+\epsilon}\Pi_{p=0}^{k-1}\Pi_{j=1}^q|(F_p)_w(H_j)|^{\epsilon/q}}
{\Pi_{j=1}^q|F(H_j)|^{\omega(j)(1-\frac{k}{m_j})}}\bigg)^{\frac{1}{h}}
|dw|.
\end{equation*}
By (\ref{ass1''}) and (\ref{eq:5}) all the conditions of Lemma \ref{ML} are satisfied. So we obtain by Lemma \ref{ML} :
\begin{equation*}
\Phi^*ds \leqslant C(\dfrac{2R}{R^2 -|w|^2})^{\rho}|dw|\,.
\end{equation*}
Since by (\ref{eq:7}) we have $0 < \rho < 1,$ it then follows that 
\begin{equation*}
d_{\Gamma_0} \leqslant \int_{\Gamma_0}ds = \int_{\overline{0,w_0}}\Phi^*ds \leqslant C \cdot \int_0^R(\dfrac{2R}{R^2 -|w|^2})^{\rho}|dw|  < + \infty,
\end{equation*}
where $d_{\Gamma_0}$ denotes the length of the divergent curve $\Gamma_0$ in $M,$ contradicting the assumption of completeness of $M.$ Claim \ref{Cl1} is proved.\\

{\bf Step 3:} We will "symmetrize" the metric $d \tau^2$ constructed in step 2 so that it will become
a complete and flat metric on $Int(A) \setminus (S \cup \tilde{S})$ (with $\tilde{S}$ another discrete subset).

We introduce a new coordinate $\xi (z) :=1/z$ 
on $A = \{z : 1/r \leq |z| < r \}$.
By (\ref{e4}) we have $S= \{ z : \Pi_{p=0}^{k}\Pi_{j=1}^q(\psi (G_z)_{jp})_z (z)= 0\} = 
\{ z : \Pi_{p=0}^{k}\Pi_{j=1}^q({\psi (G_{\xi})_{jp}})_{\xi}(z) = 0 \} $ (where the zeros are taken with the same multiplicities) and since by (\ref{ind1})
$d\tau^2$ is independent of the coordinate $z$, the change of coordinate $\xi (z) = 1/z$ yields an isometry of $A \setminus S$
onto the set $\tilde{A} \setminus \tilde{S}$, where $\tilde{A}:=\{z : 1/r < |z| \leq r \}$ 
and $\tilde{S}:= \{ z : \Pi_{p=0}^{k}\Pi_{j=1}^q({\psi (G_z)_{jp}})_z(1/z) = 0 \}$.
In particular we have 
\begin{align*}
d\tau^2&= 
\bigg(\dfrac{\Pi_{j=1}^q|G_{\xi}(H_j)(1/z)|^{\omega(j)(1-\frac{k}{m_j})}}{|((G_{\xi})_k)_{\xi}(1/z)|^{1+\epsilon}\Pi_{p=0}^{k-1}\Pi_{j=1}^q|(\psi (G_{\xi})_{jp})_{\xi}(1/z)|^{\epsilon/q}}\bigg)^{2\rho^*}
|d(1/z)|^2\\
&=
\bigg(\dfrac{\Pi_{j=1}^q|G_{z}(H_j)(1/z)|^{\omega(j)(1-\frac{k}{m_j})}}{|((G_{z})_k)_{z}(1/z)|^{1+\epsilon}\Pi_{p=0}^{k-1}\Pi_{j=1}^q|(\psi (G_{z})_{jp})_{z}(1/z)|^{\epsilon/q}}\bigg)^{2\rho^*}
|dz|^2 \ .
\end{align*}
\indent We now define 
\begin{align*}
d\tilde{\tau}^2&= \bigg(\dfrac{\Pi_{j=1}^q|G_z(H_j)(z)G_z(H_j)(1/z)|^{\omega(j)(1-\frac{k}{m_j})}}{|((G_z)_k)_z(z)((G_z)_k)_z(1/z)|^{1+\epsilon}\Pi_{p=0}^{k-1}\Pi_{j=1}^q|(\psi (G_z)_{jp})_z(z)(\psi (G_z)_{jp})_z(1/z)|^{\epsilon/q}}\bigg)^{2\rho^*}|dz|^2\\
&=\lambda^2(z)|dz|^2, 
\end{align*}
on $\tilde{A}_1 := \{z : 1/r < |z| < r \} \setminus \{ z : \Pi_{p=0}^{k}\Pi_{j=1}^q
({\psi (G_z)_{jp}})_z(z)({\psi (G_z)_{jp}})_z(1/z)  = 0 \} .$ Then $d\tilde{\tau}^2$ is complete  on $\tilde{A}_1$ : In fact by what we showed
above we have:  Towards any point of the boundary 
$\partial \tilde{A}_1 := \{z : 1/r = |z| \} \cup \{z :  |z| = r \} \cup \{ z : \Pi_{p=0}^{k}\Pi_{j=1}^q({\psi (G_z)_{jp}})_z(z)({\psi (G_z)_{jp}})_z(1/z)  = 0 \} $ of $\tilde{A}_1$, one of the factors of $\lambda^2(z)$ is bounded
from below away from zero, and 
the
other factor is the one of a complete metric with respect of this part of the boundary.
Moreover by  the corresponding properties of the two factors of $\lambda^2(z)$ it is trivial that $d\tilde{\tau}^2$ is a continuous nowhere vanishing and flat metric  on $\tilde{A}_1$.\\

{\bf Step 4 :} We produce a contradiction by using Lemma \ref{L5} to the open Riemann surface $(\tilde{A}_1, d\tilde{\tau}^2)$ : \\
In fact, we apply Lemma \ref{L5} to any point $p \in \tilde{A}_1$. Since $d\tilde{\tau}^2$ is 
complete, there cannot exist a divergent curve from $p$ to the boundary $\partial \tilde{A}_1$
with finite length with respect to $d\tilde{\tau}^2$. Since $\Phi : \Delta_{R_0} \rightarrow \tilde{A}_1$ is a local
isometry, we necessarily have $R_0 = \infty$. So $\Phi : {\mathbb C} \rightarrow \tilde{A}_1 \subset \{z : |z| <r\}$ is a non-constant holomorphic map, which contradicts to
Liouville's theorem. So our assumption (\ref{eq:2}) was wrong.
This proves the Main Theorem. 
 \end{proof}
 
  \begin{proof} (of Corollary 1 and Corollary 2)  We first observe that the inequality 
 (\ref{1}) in the Main Theorem is equivalent to the inequality
 \begin{equation} \label{1'}
  \ell (k):=\frac{k^2}{2} - k \cdot ((\sum_{j=1}^q \frac{1}{m_j})+N-\frac{1}{2}) \leq 2N-q+1\,,
  \end{equation}
where $\ell $ is a function defined on $\mathbb N \cap [1, m-1]$. Observing that 
$m-1 \leq N$, it is easy to see that the function $\ell$ is monotonely decreasing,
so if (\ref{1'}) is satisfied for some $1 \leq k \leq m-1$, it is also satisfied for $k=m-1$.
This proves Corollary 1. To prove Corollary 2, we apply the inequality (\ref{2'}) of Corollary 1 to the
$q:=\sigma_m +1$ hyperplanes $H_1,...,H_q$ assuming that $g$ meets the first $q-1$
of these hyperplanes only finitely often. Then we get $(1 - \frac{m-1}{m_q}) \leq 0$, which is
equivalent to $m_q \leq m-1$. \end{proof}

{\bf Acknowledgements.} 
A part of this work was completed during a stay of the two first  named authors at 
the Vietnam Institute for Advanced Study in Mathematics (VIASM).
The research of the second named author is partially supported by a NAFOSTED grant of Vietnam.

\vspace{1cm}

\noindent {\it Gerd Dethloff $^{1,2}$, Pham Hoang Ha$^{3}$ and Pham Duc Thoan$^{4}$\\

 \noindent $^1$ Universit\'e Europ\'eenne de Bretagne, France\\
 $^2$ Universit\'e de Brest\\
 Laboratoire de Math\'{e}matiques de Bretagne Atlantique - \\
 UMR CNRS 6205\\ 
6, avenue Le Gorgeu, BP 452\\
29275 Brest Cedex, France\\
$^3$ Department of Mathematics, Hanoi National University of Education\\
136 XuanThuy str., Hanoi, Vietnam\\
$^4$ Department of Information Technology, National University of Civil Engineering\\
55 Giai Phong str., Hanoi, Vietnam}

\noindent Email : Gerd.Dethloff@univ-brest.fr ; ha.ph@hnue.edu.vn ; ducthoan.hh@gmail.com
\end{document}